\documentclass[11pt]{amsart}
\usepackage{amsmath}
\usepackage{amscd}
\usepackage{amssymb}
\usepackage{graphicx}
\usepackage{graphics}
\usepackage{latexsym}

\input xy
\xyoption{all}

\newtheorem{theorem}{Theorem}[section]
\newtheorem{lemma}[theorem]{Lemma}
\newtheorem{proposition}[theorem]{Proposition}

\newtheorem{cor}[theorem]{Corollary}

\theoremstyle{definition}
\newtheorem{definition}[theorem]{Definition}
\newtheorem{remark}[theorem]{Remark}

\newcommand{\zr}{{\mathbb R}}

\newcommand{\zz}{{\mathbb Z}}

\newcommand{\Rspace}{{\mathbb R}}

\begin{document}

\title{Min-Max Theory for Cell Complexes}
\author{Lacey Johnson}
\email{johnsonla@ufl.edu}
%\address{Department  of Mathematics, University of Florida, P.O.~Box 118105, Gainesville, FL 32611-8105}
\author{Kevin Knudson}
\email{kknudson@ufl.edu }
\address{Department  of Mathematics,  University of Florida, Gainesville, FL 32611}

\date{\today}

\begin{abstract} In the study of smooth functions on manifolds, min-max theory provides a mechanism for identifying critical values of a function. In this paper we introduce a discretized version of this theory associated to a discrete Morse function on a (regular) cell complex. As applications we prove a discrete version of the Mountain Pass Lemma and give an alternate proof of a discrete Lusternik-Schnirelmann Theorem.
\end{abstract}

\maketitle

\section{Introduction}\label{intro} Given a smooth function $f:M\to\zr$, where $M$ is a smooth manifold, a collection of {\em min-max data} for $f$ is a pair $({\mathcal H},{\mathcal S})$, where ${\mathcal H}$ is a collection of homeomorphisms of $M$ containing functions that decrease the sublevel sets of $f$ and ${\mathcal S}$ is a collection of subsets of $M$ closed under the collection ${\mathcal H}$. Given such a pair, the min-max principle produces critical values of the function $f$.

In this paper we introduce a discretized version of this theory. While we state results for simplicial complexes, our results hold on the category of regular cell complexes. After reviewing the basics of discrete Morse theory in Section \ref{dmt}, in Section \ref{minmaxdata} we define the notion of min-max data associated to a discrete Morse function $f:K\to\zr$ on a simplicial complex $K$. The collection ${\mathcal H}$ will not be a set of homeomorphisms; rather it consists of functions defined on the power set of the set of cells of $K$. There are various operators associated to the function $f$ which yield the requisite property of decreasing the level sets, and so we are able to obtain a discrete min-max principle which allows us to identify critical values of $f$.

As applications of this, we deduce the following. First, we obtain a discrete version of the Mountain Pass Lemma: given two local minima for $f$, there is a critical edge $e$ ``between" them; see Section \ref{sec:mtnpass} (a proof of the corresponding smooth result may be found in \cite{nico}, p.~89). We also give an alternate proof of a discrete Lusternik-Schnirelmann Theorem as first presented in \cite{scoville}; this is the content of Section \ref{lscat}.

\subsection*{Acknowledgements} The second author spoke about some of these results at the International Conference on Algebra and Related Topics in Rabat, Morocco, July 2018. He thanks the organizers, Driss Bennis in particular, for the invitation and for an enjoyable and stimulating conference.

\section{Discrete Morse Theory}\label{dmt}

Let $K$ be any finite simplicial complex, where $K$ need not be a triangulated manifold nor have any other special property~\cite{forman}. When we write $K$ we mean the set of simplices of $K$; by $|K|$ we mean the underlying topological space. 
Let $\alpha^{(p)} \in K$ denote a simplex of dimension $p$. 
Let $\alpha < \beta$ denote that simplex $\alpha$ is a face of simplex $\beta$. If $f:K\to \Rspace$ is a function
define $U(\alpha) = \{\beta^{(p+1)} > \alpha \mid f(\beta) \leq f(\alpha) \}$ 
and $L(\alpha) = \{\gamma^{(p-1)} < \alpha \mid f(\gamma) \geq f(\alpha) \}$.
In other words, $U(\alpha)$ contains the immediate cofaces of $\alpha$ with lower (or equal) function values, while $L(\alpha)$ contains the immediate  faces of $\alpha$ with higher (or equal) function values. 
Let $|U(\alpha)|$ and $|L(\alpha)|$ be their sizes.  

\begin{definition}
A function $f: K \to \Rspace$ is a \emph{discrete Morse function} if for every $\alpha^{(p)} \in K$, 
(i) $|U(\alpha)| \leq 1$ and 
(ii) $|L(\alpha)| \leq 1$.
\end{definition}

Forman showed that conditions (i) and (ii) are exclusive -- if one of the sets $U(\alpha)$ or $L(\alpha)$ is nonempty then the other one must be empty (\cite{forman}, Lemma 2.5).
Therefore each simplex $\alpha \in K$ can be paired with at most one exception simplex: either a face $\gamma$ with larger function value, or a coface $\beta$ with smaller function value. 
We refer to conditions (i) and (ii) as the \emph{Morse conditions}. 

\begin{definition}
A simplex $\alpha^{(p)}$ is \emph{critical} if (i) $|U(\alpha)| = 0$ and (ii) $|L(\alpha)| = 0$. A {\em critical value} of $f$ is its value at a critical simplex. 
\end{definition}

\begin{definition}
A simplex $\alpha^{(p)}$ is \emph{regular} if either of the following conditions holds: (i) $|U(\alpha)| = 1$; (ii) $|L(\alpha)| = 1$; as noted above these conditions cannot both be true (\cite{forman}, Lemma 2.5). 
\end{definition}

Note that regular simplices occur in pairs. Suppose there is a pair $\sigma^{(p)}<\tau^{(p+1)}$ of simplices in $K$ such that $\sigma$ has no other cofaces in $K$. Then $K\setminus\{\sigma,\tau\}$ is a simplicial complex called an {\em elementary collapse} of $K$. We say that $K$ {\em collapses to} $L$ if $L$ can be obtained from $K$ by a sequence of elementary collapses. We denote this by $K\searrow L$. A complex $K$ is {\em collapsible} if $K\searrow \{v\}$ for some vertex $v$.

Two Morse functions $f,g:K\to\zr$ are {\em equivalent} if for every pair $\sigma^{(p)}<\tau^{(p+1)}$ in $K$, $f(\sigma)< f(\tau)$ if and only if $g(\sigma)< g(\tau)$. A proof of the following is embedded in the proof of Theorem 3.3 of \cite{forman}.

 \begin{lemma}\label{excellent} Let $f:K\to\zr$ be a discrete Morse function. Then there is an injective discrete Morse function $g:K\to\zr$ equivalent to $f$ having the same critical simplices as $f$. \hfill $\qed$
 \end{lemma}

\subsection*{Main results of DMT}
Given $c\in \Rspace$, we have the \emph{level subcomplex} $K^c = \cup_{f(\alpha) \leq c} \cup_{\beta \leq \alpha} \beta$. 
That is, $K^c$ contains all simplices $\alpha$ of $K$ such that $f(\alpha) \leq c$ along with all of their faces. 
We have the following two combinatorial versions of the main results of classical Morse theory. The fundamental idea in their proofs is to collapse the sublevel complexes by removing pairs $\{\sigma<\tau\}$ in the discrete gradient vector field of the function $f$.

\begin{theorem}[DMT Part A,~\cite{forman}] 
\label{theorem:dmt-a}
Suppose the interval $(a,b]$ contains no critical value of $f$. 
Then $K^b$ is homotopy equivalent to $K^a$. 
In fact, $K^b$ collapses onto $K^a$. \hfill $\qed$
\end{theorem}

The next theorem explains how the topology of the sublevel complexes changes as one passes a critical value of a discrete Morse function. In what follows, $\dot{e}^{(p)}$ denotes the boundary of a $p$-simplex $e^{(p)}$. 

\begin{theorem}[DMT Part B,~\cite{forman}] 
\label{theorem:dmt-b}
Suppose $\sigma^{(p)}$ is a critical simplex with $f(\sigma) \in (a,b]$, and there are no other critical simplices with values in $(a,b]$. Then $K^b$ is homotopy equivalent to attaching a $p$-cell $e^{(p)}$ along its entire boundary in $K^a$; that is, $K^b= K^a \cup_{\dot{e}^{(p)}} e^{(p)}$. \hfill $\qed$
\end{theorem}

\subsection*{The associated gradient vector field} Given a discrete Morse function $f:K\to\Rspace$ we may associate a discrete gradient vector field as follows. Since any noncritical simplex $\alpha^{(p)}$ has at most one of the sets $U(\alpha)$ and $L(\alpha)$ nonempty, there is a unique face $\nu^{(p-1)}<\alpha$ with $f(\nu)\ge f(\alpha)$ or a unique coface $\beta^{(p+1)}>\alpha$ with $f(\beta)\le f(\alpha)$. Denote by $V$ the collection of all such pairs $\{\sigma <\tau\}$ (referred to as \emph{Morse pairs}). Then every simplex in $K$ is in at most one pair in $V$ and the simplices not in any pair are precisely the critical cells of the function $f$. We call $V$ the {\em gradient vector field associated to $f$}.  We visualize $V$ by drawing an arrow from $\alpha$ to $\beta$ for every pair $\{\alpha<\beta\}\in V$. Observe that $V$ ``points" in the direction of {\em decrease} for $f$; as such it corresponds to the {\em negative} gradient of the function and we may write $V = -\nabla f$. Theorems \ref{theorem:dmt-a} and \ref{theorem:dmt-b} may then be visualized in terms of $V$ by collapsing the pairs in $V$ using the arrows. Thus a discrete gradient (or equivalently a discrete Morse function) provides a collapsing order for the complex $K$, simplifying it to a complex $L$ with potentially fewer cells but having the same homotopy type.

The collection $V$ has the following property. By a \emph{$V$-path}, we mean a sequence $$\alpha^{(p)}_0 < \beta_0^{(p+1)}>\alpha_1^{(p)}<\beta_1^{(p+1)}>\cdots <\beta_r^{(p+1)}>\alpha_{r+1}^{(p)}$$ where each $\{\alpha_i<\beta_i\}$ is a pair in $V$. Such a path is {\em nontrivial} if $r>0$ and {\em closed} if $\alpha_{r+1}=\alpha_0$. Forman proved the following result.

\begin{theorem}[\cite{forman}] If $V$ is a gradient vector field associated to a discrete Morse function $f$ on $K$, then $V$ has no nontrivial closed $V$-paths.
\end{theorem}

\subsection*{Ascending and descending regions} If $g:M\to\zr$ is a Morse function on a smooth compact manifold $M$ and if $p$ is a critical point of $g$, there are two subspaces associated to $p$: the ascending and descending discs. The former is the set of points $q$ in $M$ such that the integral curve of $-\nabla g$ passing through $q$ converges to $p$ as $t\to \infty$ while the latter is the set of points whose integral curves converge to $p$ as $t\to-\infty$. If $p$ has index $\lambda$, then the ascending disc is topologically a disc of dimension $n-\lambda$ and the descending disc has dimension $\lambda$. 

The discrete analogues of these objects were constructed by Jer\v{s}e and Mramor Kosta in \cite{jerse-kosta}. The idea is that descending discs should be unions of $V$-paths, just as in the smooth case (where descending discs are foliated by integral curves of a gradient-like vector field). Ascending discs are then constructed as descending discs for the dual vector field $V^\ast$ on the dual cell complex $K^\ast$. 

For our purposes, however, we need only consider ascending discs of local minima and for this it suffices to consider the {\em basin} of a minimum $v$, defined in \cite{robins}. This is the maximal subcomplex of $K$ that collapses to $v$; we denote it by $A(v)$.

\section{Min-max data associated to a discrete Morse function}\label{minmaxdata}

Let $K$ be a finite simplicial complex and let $f:K\to\zr$ be a discrete Morse function (which we may assume injective). Recall that for $a\in\zr$, $K^a$ denotes the sublevel complex associated to $a$. We denote by $L^a$ the {\em sublevel set} $L^a=\{\sigma\in K: f(\sigma)\le a\}$. Note that $K^a$ is the closure of $L^a$.

Denote by ${\mathcal P}(K)$ the power set of $K$. 

\begin{definition}\label{minmaxdef} A collection of {\em min-max data} for the discrete Morse function $f:K\to\zr$ is a pair $({\mathcal H},{\mathcal S})$ satisfying the following conditions.
\begin{enumerate}
\item ${\mathcal H}$ is a collection of maps ${\mathcal P}(K) \to {\mathcal P}(K)$ such that for every regular value $a$ of $f$ there exist $\varepsilon >0$ and $h\in {\mathcal H}$ such that $$h(L^{a+\varepsilon})\subset L^{a-\varepsilon}.$$

\item ${\mathcal S}$ is a collection of subsets of $K$ such that for all $h\in {\mathcal H}$ and all $S\in {\mathcal S}$, $h(S)\in {\mathcal S}$.
\end{enumerate}
\end{definition}

\begin{theorem}\label{minmaxthm} If $({\mathcal H},{\mathcal S})$ is a collection of min-max data for the discrete Morse function $f$ on $K$ then the number $$c = \min_{S\in{\mathcal S}}\max_{\sigma\in S} f(\sigma)$$ is a critical value for $f$.
\end{theorem}

\begin{proof} Suppose not; that is, assume that $c$ is a regular value for $f$. Then there exist $\varepsilon >0$ and $h\in {\mathcal H}$ such that $$h(L^{c+\varepsilon})\subset L^{c-\varepsilon}.$$ By the definition of $c$ there is an $S\in {\mathcal S}$ such that $$\max_{\sigma\in S} f(\sigma) < c+\varepsilon;$$ that is, $S\subset L^{c+\varepsilon}$. Then $S' = h(S) \in {\mathcal S}$ and $h(S)\subset L^{c-\varepsilon}$. It follows that $$\max_{\sigma\in S'} f(\sigma) \le c-\varepsilon,$$ which implies $$\min_{S'\in{\mathcal S}}\max_{\sigma\in S'} f(\sigma) \le c-\varepsilon,$$ contrary to the choice of $c$.
\end{proof}

\section{Mountain Pass Lemma}\label{sec:mtnpass}

Our goal in this section is to prove the Mountain Pass Lemma in the setting of an arbitrary simplicial complex. Recall the classical smooth situation. Given a Morse function $f:M\to\zr$  on the compact manifold $M$, suppose $f$ has two local minima, say at $x_0$ and $x_1$ with $f(x_0)<f(x_1)$. Assuming $M$ is connected, there is a path from $x_0$ to $x_1$. The Mountain Pass Lemma asserts that the infimum over all such paths $\gamma$ of the maximum value $\sup_t \{\gamma(t)\}$ is a critical value of $f$. 

The picture to have in mind is shown in Figure \ref{mountainpassfig}. Thinking of $f$ as a height function, the minima at $x_0$ and $x_1$ represent depressions in the landscape. To get from one low point to another, one must cross over a ridge; the lowest point $z$ on the ridge will be a saddle point; i.e. a critical point of $f$.

While this result may seem obvious, it is not at all clear how one could prove it on a general $n$-manifold, let alone on a general simplicial complex. Our plan is to construct an appropriate collection of min-max data and then apply Theorem \ref{minmaxthm}.

\begin{figure}
\centerline{\includegraphics[height=2.5in]{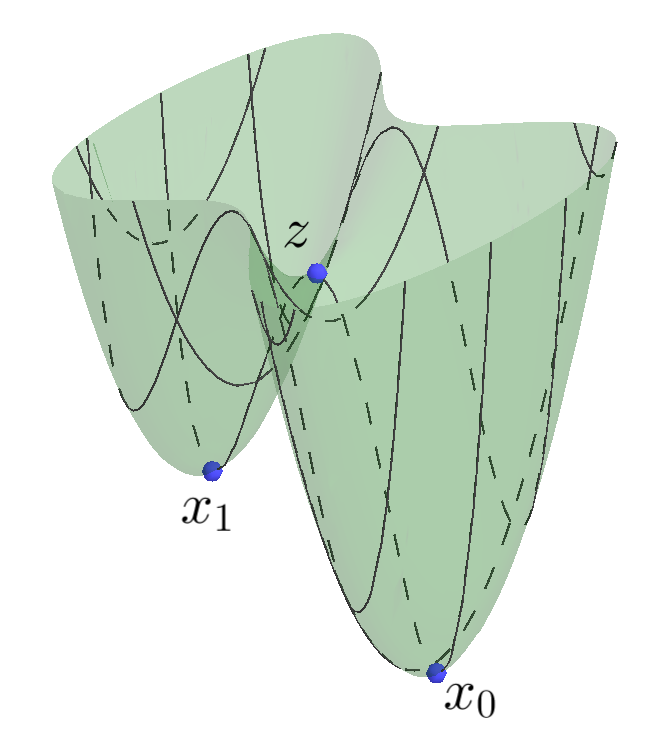}}
\caption{\label{mountainpassfig} Visualizing mountain pass points.}
\end{figure}

Suppose $V$ is a discrete gradient vector field on $K$; say $V = -\nabla f$ for some discrete Morse function $f$. Denote by $\langle - , - \rangle$ the inner product on the chains $C_p(K,\zz)$ defined by making all cells orthogonal and define a map $V: C_p(K,\mathbb{Z} )\rightarrow C_{p+1}(K,\mathbb{Z})$ as follows. Suppose $\sigma$ is a $p$-cell with fixed orientation. If there is a $(p+1)$-cell $\tau$ with $\{\sigma <\tau \}\in V$, then set $$V(\sigma )= -\langle\partial\tau , \sigma \rangle\tau,$$ where $\partial$ is the boundary map in $K$. If there is no such $\tau$, (for example, if $\sigma$ is critical), then set $V(\sigma ) = 0$. We extend this linearly to a map on all of $C_p(K)$.

\begin{definition} The {\em discrete flow operator} $\phi:C_p(K)\to C_p(K)$  is defined by $$\phi =\text{Id} + \partial V + V\partial.$$
\end{definition}

The intuition here is that if $v$ is a vertex then either (a) $v$ is critical and should be fixed by the flow, or (b) $v$ is not critical and it should flow to the other vertex of the edge $e$ is paired with by the gradient $V$; that is, $V(v) = \pm e = v + \partial V(v)$.

The following is Theorem 6.4 of \cite{forman}.

\begin{proposition}\label{flowgoesdown} Let $\sigma_1,\dots,\sigma_r$ denote the $p$-cells of $K$, each with a chosen orientation. Write $$\phi(\sigma_i) = \sum_j a_{ij}\sigma_j.$$ Then
\begin{enumerate}
\item For every $i$, $a_{ii}=0$ or $1$, and $a_{ii}=1$ if and only if $\sigma_i$ is critical.
\item If $i\ne j$, then $a_{ij}\in\zz$. If $i\ne j$ and $a_{ij}\ne 0$, then $f(\sigma_j)<f(\sigma_i)$. \hfill $\qed$
\end{enumerate}
\end{proposition}

For our purposes, we will need to consider the collection of cells that appear in the image of $\phi$. Note that we may apply ${\phi}$ to sets of $p$-cells in $K$ to obtain a map, denoted ${\Phi}: {\mathcal P}(K)\to {\mathcal P}(K)$, defined by $$\Phi(A) = \bigcup_{\sigma\in A} \phi(\sigma).$$ That is, for a cell $\sigma$ we are using the notation $\phi(\sigma)$ to mean the collection of cells appearing in the chain $\phi(\sigma)$.

We also define a second map $\overline{\Phi}$ on ${\mathcal P}(K)$ by $$\overline{\Phi}(A) = \text{subcomplex of $K$ generated by $\Phi(A)$}.$$ We give examples of $\Phi$ and $\overline{\Phi}$ in Figure \ref{phiexamples}.

\begin{figure}
    \centering
    \includegraphics[width=1.65in]{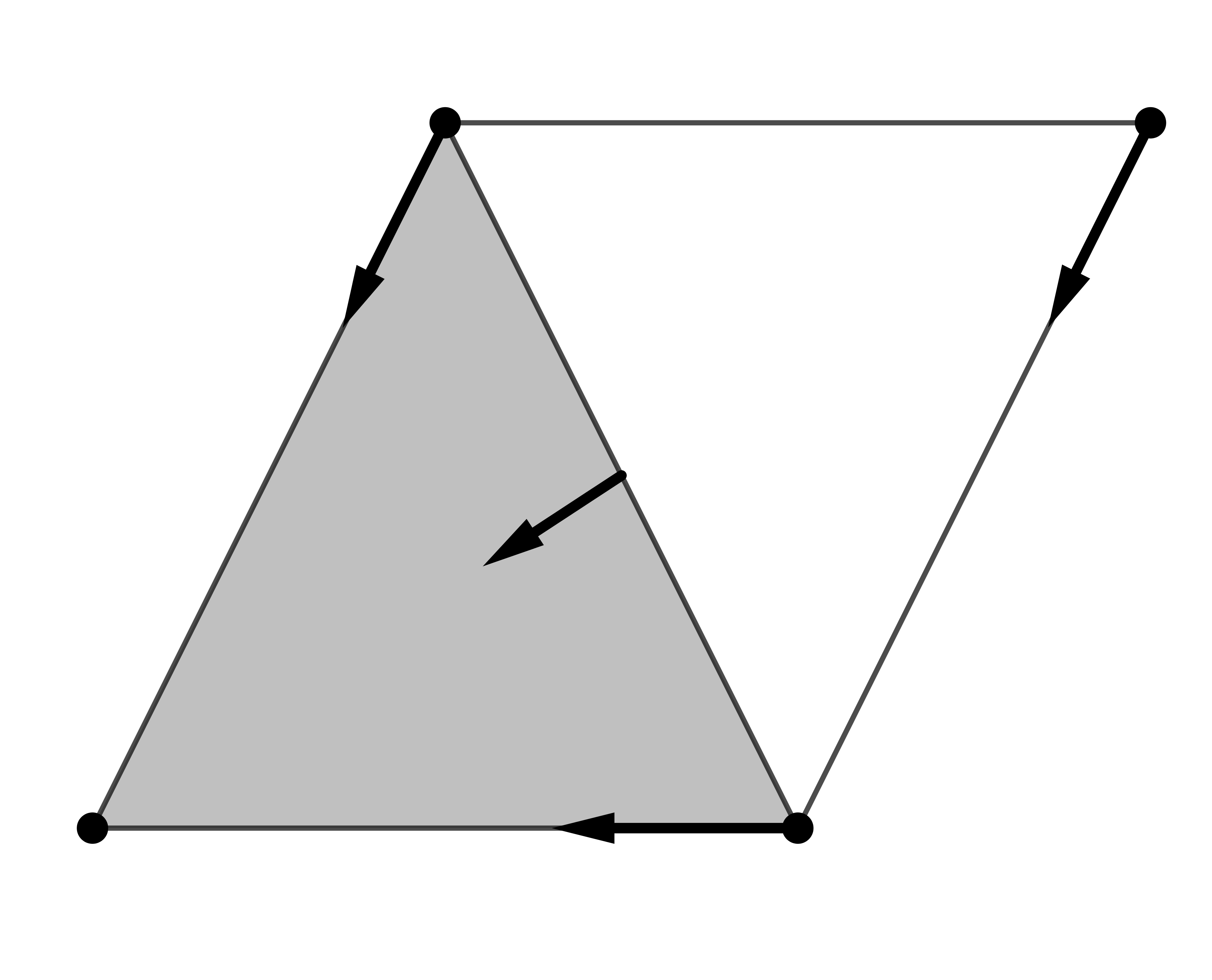}\includegraphics[width=1.65in]{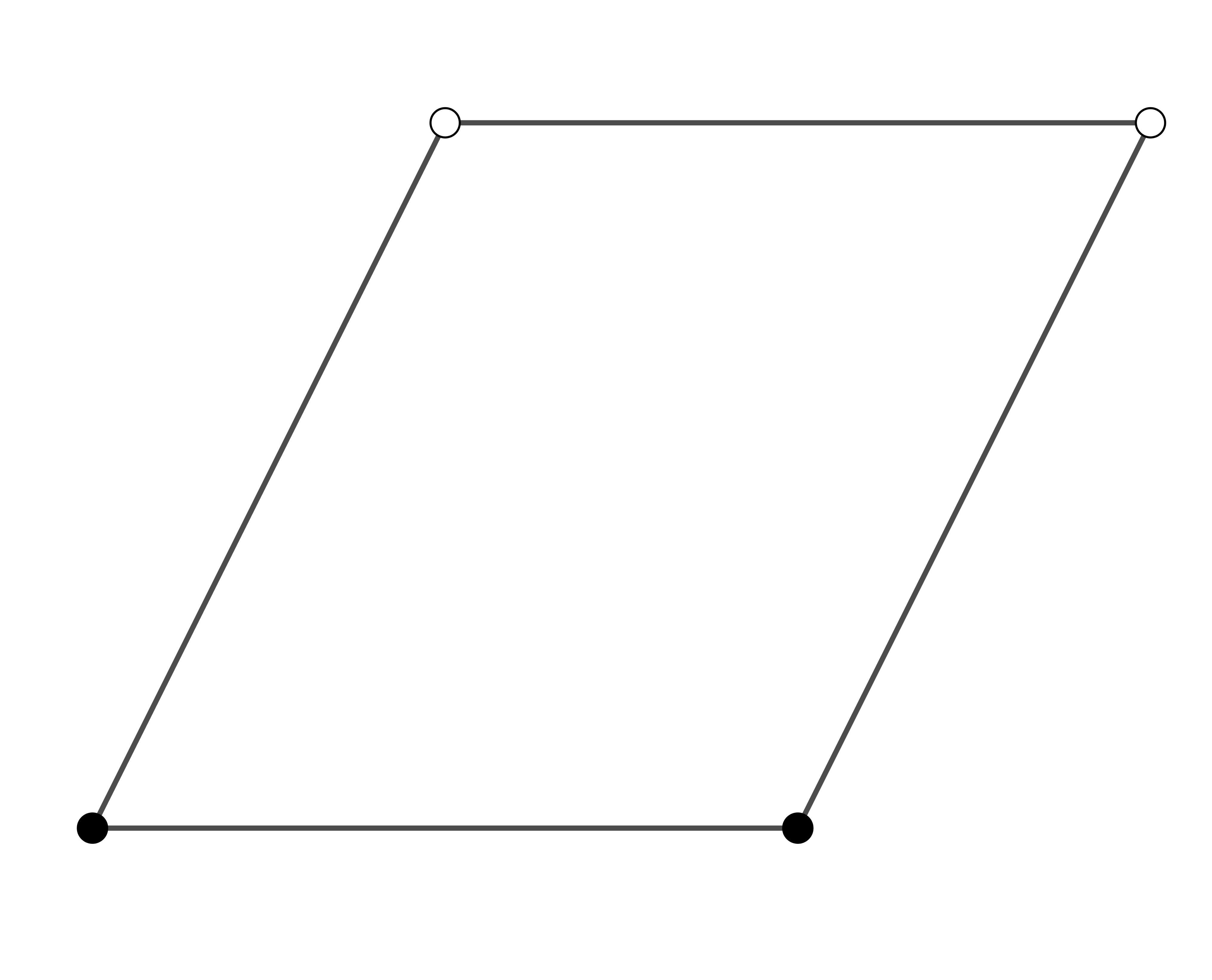}\includegraphics[width=1.65in]{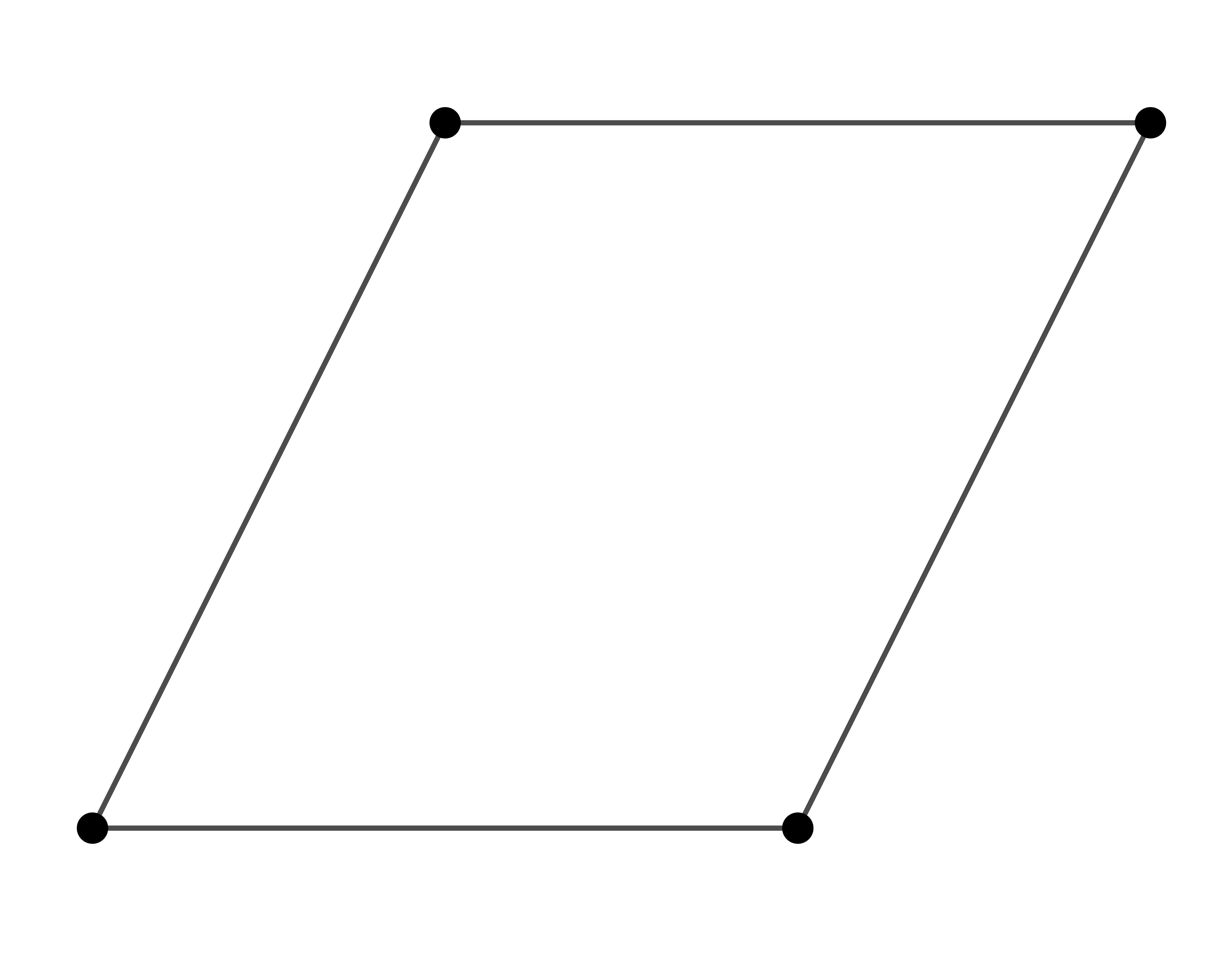}
    \caption{A complex $K$ (left), $\Phi(K)$ (center), $\overline{\Phi}(K)$ (right).}
    \label{phiexamples}
\end{figure}

The maps $\Phi$ and $\overline{\Phi}$ will be useful for constructing min-max data.  Consider $\mathcal{H} = \{ {\Phi} \}$. By Proposition \ref{flowgoesdown}, condition (1) in the definition of min-max data for a discrete Morse function is satisfied.  In fact, for any $a$, regular or not, ${\Phi}(K^a)\subseteq K^a$ (see, e.g. Lemma 1 of \cite{robins}), but ${\Phi}$ only decreases the sublevel set for regular values. The same is true for the map $\overline{\Phi}$. In fact we have the following, which is a special case of Theorem 4 of \cite{robins}.

\begin{lemma}\label{barphicollapse} For each $a\in\zr$, $K^a\searrow \overline{\Phi}(K^a)$.
\end{lemma}

\begin{proof}
Note that any critical simplex in $K^a$ is also in $\overline{\Phi}(K^a)$, and hence so are its faces. So a simplex $\sigma$ in $K^a\setminus\overline{\Phi}(K^a)$ must be regular and paired with some $\tau\in K^a$. Order such pairs by decreasing $f$-value; the maximal pair must be of the form $\{\alpha<\beta\}$ where $\alpha$ is a free face of $\beta$ (note that if $a$ is a regular value then this pair consists of the unique simplex in $f^{-1}(a)$ and either a coface of lower value or face of higher value). Removing the pairs in decreasing order yields a collapsing sequence $K^a\searrow \overline{\Phi}(K^a)$.
\end{proof}

Now suppose that $v_0$ and $v_1$ are critical vertices for the discrete Morse function $f$ and assume $f(v_0)<f(v_1)$. Let $\mathcal{S}$ be the collection of subsets of $K$ consisting of sets of the form $E=\{v_1,e_1,e_2,\dots,e_r\}$, where the $e_i$ form an edge path beginning at $v_1$ and ending in the basin $A(v_0)$. Writing $e_i=<u_{i-1},u_i>$, we insist on the following condition: if $u_j$ belongs to the ascending disc of $v_0$ then $f(e_j)>f(e_{j+1})>\cdots >f(e_r)$. This condition is not strictly necessary, but it does prevent path bifurcation when acted upon by ${\Phi}$.  We may assume that such paths do not pass through another critical vertex $v\ne v_0,v_1$.

Note that if $E\in {\mathcal S}$, ${\Phi}(E)\ne\emptyset$. Indeed, since $v_1$ is a critical vertex, ${\Phi}(v_1)=v_1$.

\begin{theorem}\label{condition2mtnpass} If $E\in\mathcal{S}$, then ${\Phi} (E) \in \mathcal{S}$. 
\end{theorem}

Assuming this we see that $(\mathcal{H},\mathcal{S})$ is a collection of min-max data on $K$. By Theorem \ref{minmaxthm} we see that 
$$c = \min_{E\in {\mathcal S}}\max_{e\in E} f(e)$$ is a critical value of $f$. This critical value must be greater than $f(v_1)$ since any set $E$ in $\mathcal S$ has the form $E=\{v_1,e_1,\dots,e_r\}$ with $r\ge 1$ and $f(e_1)>f(v_1)$. Thus, among all the edge paths from $v_1$ to $v_0$, there is one passing through a critical edge $e$ and the value of $f(e)$ is minimal among all such paths. Thus, $e$ may be thought of as the ``saddle point" along the ``ridge" between $v_0$ and $v_1$.

We have thus proved the discrete version of the Mountain-Pass Lemma and we refer to the critical edges that result as mountain-pass edges.  The Mountain-Pass Lemma implies that if a discrete Morse function has two vertices that are strict local minima, then it must admit a critical edge.

Another explanation of why there should exist a critical edge with function value greater than $f(v_1)$ is to observe that the sublevel complex $K^{f(v_1)}$ is disconnected while the cell complex $K$ is connected.  The change in topological type in going from $K^{f(v_1)}$ to $K$ can be explained by the presence of a critical edge with function value greater than $f(v_1)$.

\begin{remark}
A slightly different version of this result was proved in \cite{robins}; the authors did not employ min-max theory to do so. 
\end{remark}

\subsubsection*{Proof of Theorem \ref{condition2mtnpass}} We begin by computing the action of ${\Phi}$ on a single edge $e=<v,w>$. There are three cases.
\begin{enumerate}
\item Suppose $e$ is critical. Then $\phi(e) = e+V\partial e = e-V(v)+V(w)$. It follows that ${\Phi}(e) = \{e,V(v),V(w)\}$.
\item If $e$ is regular and paired with one of its vertices (say $v$), then $\phi(e) = e+V\partial e = V(w)$. Thus $\Phi(e) = \{V(w)\}$.
\item If $e$ is regular and paired with the 2-simplex with edges $e,e',e''$, then $\phi(e) = e+V\partial e+\partial V(e)$. Several things could happen in this case depending on how the vertices $v$ and $w$ are paired (or not) by $V$. We will evaluate this scenario as needed in what follows. 
\end{enumerate}

We proceed by induction on $r$, the number of edges in a set $E$. The case $r=1$ is essentially described above. Here we have $E=\{v_1,e\}$ for some edge $e$. The other vertex $u$ of $e$ lies in the ascending disc of $v_0$. If $e$ is critical we then have $\phi(e) = e + V(u)$ and so ${\Phi}(E) = \{v_1,e,V(u)\}$ is another path ending in $A(v_0)$. Since $u$ is in $A(v_0)$ it cannot be paired with $e$ and so the second case above cannot occur. This leaves the final possibility that $e$ is regular and paired with a 2-simplex $<v_1,u,w>$. In this case we see that $\phi(e) = <v_1,w>+<u,w> + V(u)$. It could be that $V(u) = -<u,w>$, which implies that $w$ also belongs to $A(v_0)$ and hence ${\Phi}(E) = \{v_1,<v_1,w>\}$ is another element of $\mathcal S$. Otherwise, $V(u)$ is another edge ending in $A(v_0)$ and so ${\Phi}(E) = \{v_1,<v_1,w>,<u,w>,V(u)\}\in {\mathcal S}$.

Now assume the result holds for all paths of length $< r$ and suppose $E=\{v_1,e_1,\dots,e_r\}\in {\mathcal S}$. If $\{v_1,e_1,\dots,e_{r-1}\}$ ends in $A(v_0)$, then $\phi(e_r)$ also lands in $A(v_0)$ and by the induction hypothesis, $\Phi(E)\in {\mathcal S}$. Now suppose that $\{v_1,e_1,\dots,e_{r-1}\}$ does not end in $A(v_0)$. This means that there is no gradient path $$u_{r-1}<\epsilon_0>w_1<\epsilon_1>\cdots<\epsilon_s>v_0.$$ However, there is a path $$\eta_0>u_r<\eta_1>\nu_1<\eta_2>\cdots<\eta_t>v_0.$$ It follows that $e_r\ne\eta_0$ and $f(e_r)>f(u_r)$. Note that this implies that $u_{r-1}$ could not be paired with $e_r$ (else $f(u_r)<f(e_r)$ which would imply that $u_{r-1}\in A(v_0)$). Thus, $e_r$ is paired with a $2$-simplex $\sigma$ (whose faces are $e_r$, $e'$, and $e''$), or $e_r$ is critical. In the first case, $\phi(e_r)=e'+e''+V(u_r)+V(u_{r-1})$ simply diverts the path $E$ around the edges of $\sigma$ and so ${\Phi}(E)\in{\mathcal S}$. If $e_r$ is critical then $\phi(e_r)=e_r+V(u_{r-1})+V(u_r)$ and so ${\Phi}(E)$ still lies in ${\mathcal S}$. This completes the proof. \hfill $\qed$

\begin{remark}
The reader might wonder why we used the map $\Phi$ and edge paths that do not include the vertices as opposed to the map $\overline{\Phi}$ and $1$-dimensional subcomplexes. The issue is that arbitrary edge paths emanating from $v_1$ might contain vertices that are paired with edges not in the path. Applying $\Phi$ or $\overline{\Phi}$ to such an object leads to subcomplexes that have extraneous edges jutting from the path. We would then be forced to expand the collection ${\mathcal S}$ to accommodate this, causing some technical difficulties.
\end{remark}

\section{A discrete Lusternik-Schnirelmann theorem}\label{lscat}
Suppose $K$ is a simplicial complex and that $L$ is a subcomplex. Recall that $L$ is {\em collapsible} if it collapses to a vertex. The following definition appears in \cite{scoville}.

\begin{definition} Suppose that $L$ is a subcomplex of $K$. Then $L$ has {\em geometric precategory less than or equal to $m$ in $K$}, denoted $\widetilde{\text{dgcat}}_K(L)\le m$ if there exist $m+1$ subcomplexes $U_0,\dots, U_m$ in $K$, each of which is collapsible in $K$, such that $L\subseteq\bigcup_{i=0}^m U_i$. If $\widetilde{\text{dgcat}}_K(L)\not< m$ we say $\widetilde{\text{dgcat}}_K(L) = m$. The {\em discrete {\em (}geometric{\em )} category of $L$ in $K$} is then $\text{dgcat}_K(L) = \min\{\widetilde{\text{dgcat}}_K(L'): L\searrow L'\}$.
\end{definition}

For each $k$, let $$\Gamma_k=\{L\subset K: \exists K^a\searrow L\;\text{with}\; \text{dgcat}_K(K^a)\ge k-1\}.$$ Recall that if $A\in{\mathcal P}(K)$, we define $\overline{\Phi}(A)$ to be the subcomplex of $K$ generated by $\Phi(A)$.

\begin{theorem} For each $k$, $(\{\overline{\Phi}\},\Gamma_k)$ is a collection of min-max data.
\end{theorem}

\begin{proof}
Suppose $L\in\Gamma_k$. Then there is some $a\in\zr$ with $K^a\searrow L$. Note that $\overline{\Phi}(K^a)\searrow \overline{\Phi}(L)$ and since Lemma \ref{barphicollapse} implies $K^a\searrow \overline{\Phi}(K^a)$  we see that $\overline{\Phi}(L)\in\Gamma_k$.
\end{proof}

\begin{cor}\label{minmaxcor}
For each $k$ the real number $\displaystyle c_k=\min_{L\in\Gamma_k}\max_{\sigma\in L}f(\sigma)$ is a critical value of $f$. \hfill $\qed$
\end{cor}

As a consequence of this, we obtain the following Lusternik-Schnirelmann theorem, originally proved in \cite{scoville} (Theorem 26).

\begin{cor}\label{lusternikthm}
Let $f:K\to\zr$ be a discrete Morse function with $m$ critical values. Then $\text{\em dgcat}(K)+1\le m$.
\end{cor}

\begin{proof}
By Corollary \ref{minmaxcor} each $c_k$ is a critical value of $f$. We may assume that the global minimum of $f$, which occurs at some vertex $v$ is $f(v)=0$. Then $c_1=0$ and $K^{c_1}$ contains one critical simplex. Proceeding inductively, suppose $K^{c_{k-1}}$ has at least $k-1$ critical simplices and consider $K^{c_k}$. Since $f$ is excellent, $c_{k-1}<c_k$ and so $f^{-1}(c_k)$ contains at least one new critical simplex. Thus $K^{c_k}$ contains at least $k$ critical simplices. So if $c_1<c_2<\cdots <c_{\text{dgcat}(K)+1}$ are the distinct critical values then $K^{c_{\text{dgcat}(K)+1}}\subseteq K$ contains at least $\text{dgcat}(K)+1$ critical simplices. Therefore $\text{dgcat}(K)+1\le m$.
\end{proof}

\end{document}